\documentclass[12pt]{amsart}
\usepackage[english]{babel}
\usepackage[applemac]{inputenc}
\usepackage[T1]{fontenc}

\usepackage{bbm}
\usepackage{float, verbatim}
\usepackage{amsmath}
\usepackage{amssymb}
\usepackage{amsthm}
\usepackage{amsfonts}
\usepackage{graphicx}
\usepackage{mathtools}
\usepackage{enumitem}
\usepackage[pagebackref]{hyperref}
\usepackage{color}
\usepackage{tikz}
\usepackage{xcolor}
\usepackage{pgfplots}
\usepackage{caption}
\usepackage{subcaption}
\usepackage{enumerate}
\usepackage{url}

\hypersetup{
	colorlinks=true,
	linkcolor=blue,
	filecolor=magenta,      
	urlcolor=cyan,
	citecolor= red,
}

\usepackage[final]{showlabels}

\newcommand{\Haus}{\dim_{\mathrm{H}}}

\newtheorem*{thm*}{Theorem}
\newtheorem*{conj*}{Conjecture}
\newtheorem*{ques*}{Question}
\newtheorem*{rem*}{Remark}
\newtheorem*{defn*}{Definition}
\newtheorem*{mainques*}{Main questions}

\newtheorem{thmx}{Theorem}

\newtheorem{thm}{Theorem}[section]

\newtheorem{lma}[thm]{Lemma}

\newtheorem{defn}[thm]{Definition}

\newtheorem{conj}[thm]{Conjecture}
\newtheorem{rem}[thm]{Remark}

\def\supp{\mathrm{supp}}

\def\cl{\mathrm{cl}}

\pgfplotsset{compat=1.16}

\begin{document}
\title[radial and linear projection]{On the absolute continuity of radial and linear projections of missing digits measures}

\author{Han Yu}
\address{Han Yu, Mathematics Institute, Zeeman Building, University of Warwick, Coventry CV4 7AL, UK}
\curraddr{}
\email{han.yu.2@warwick.ac.uk}



\maketitle

\begin{abstract}
In this paper, we study the absolute continuity of radial projections of missing digits measures. We show that for large enough missing digits measures $\lambda$ on $\mathbb{R}^n,n\geq 2,$ for all $x\in\mathbb{R}^n\setminus \supp(\lambda),$ $\Pi_x(\lambda)$ is absolutely continuous with a density function in $L^2(S^{n-1}).$ Our method applies to linear projections as well. In particular, we show that for $\lambda$ as above, the linearly projected measure $P_\theta(\lambda)$ is absolutely continuous with a continuous density function for almost all directions $\theta\in S^{n-1}.$ This implies a version of Palis' conjecture for missing digits sets.
\end{abstract}

\maketitle
\allowdisplaybreaks

\section{Introduction}
\subsection{Radial Projection}
Let $n\geq 2$ be an integer. Let $x\in\mathbb{R}^n.$ Define the function
\[
\Pi_x: y\in\mathbb{R}^n\setminus\{x\}\to \Pi_x(y)=\frac{y-x}{|y-x|}\in \mathbb{S}^{n-1}.
\]
This function $\Pi_x$ is the radial projection centred at $x$. Let $A\subset\mathbb{R}^n.$ The set $\Pi_x(A\setminus\{x\})$ is the radial projection of $A$ from the observing position $x$. Let $\lambda$ be a compactly supported Borel probability measure on $\mathbb{R}^n.$ Suppose that $x$ is not in $\supp(\lambda).$ Then $\Pi_x(\lambda)$ is the pushed forward measure. See Section \ref{sec: rad}. The following result was proved in \cite[Theorem 1.11]{O19}.
\begin{thm}\label{thm: Or1}
    Let $n\geq 2$ be an integer. Let $\lambda$ be a compactly supported Borel probability measure on $\mathbb{R}^n.$ Suppose that
    \[
    I_s(\lambda)=\int\int\frac{d\lambda(x)d\lambda(y)}{|x-y|^s}<\infty
    \]
    for some $s>n-1.$ Then $\Pi_x(\lambda)$ is absolutely continuous for all $x\in\mathbb{R}^n\setminus\supp(\lambda)$ outside of an exceptional set with Hausdorff dimension at most $2(n-1)-s.$  
\end{thm}
For example, if $\lambda$ is $AD$-regular (see Section \ref{sec: AD}), then the finiteness of the integral happens when $s>\Haus \lambda.$ Roughly speaking, this result tells us that for thick enough measures, the radial projections are absolutely continuous for most of the 'observing positions'. The projected measures can have more regularities than just being absolutely continuous. In fact, the following finer result holds. See \cite[Theorem 1.13]{O19}.
\begin{thm}\label{thm: Or2}
    Let $n\geq 2$ be an integer. Let $\lambda$ be a compactly supported Borel probability measure on $\mathbb{R}^n.$ Suppose that
    \[
    I_s(\lambda)=\int\int\frac{d\lambda(x)d\lambda(y)}{|x-y|^s}<\infty
    \]
    for some $s>n-1.$ Let $p>1$. Then $\Pi_x(\lambda)$ is absolutely continuous with a density function in $L^p$ for all $x\in\mathbb{R}^n\setminus\supp(\lambda)$ outside of an exceptional set with Hausdorff dimension at most $2(n-1)-s+\delta(p)$ where $\delta(p)>0$ and moreover $\delta(p)\downarrow 0$ as $p\downarrow 1.$ 
\end{thm}
For general $\lambda$, it is very difficult, if at all possible, to obtain more information about the exceptional sets in Theorem \ref{thm: Or1} and Theorem \ref{thm: Or2}. However, if $\lambda$ satisfies some other restrictions, e.g. being self-similar, then it is likely that the exceptional set for $\lambda$ is empty. In this paper, we consider radial projections of missing digits measures. See Section \ref{sec: missing}. In this case, it is plausible that the following conjecture holds.
\begin{conj}\label{conj}
    Let $n\geq 2$. Let $\lambda$ be a missing digits measure (or a Cartesian product of missing digits measures). Suppose that $\Haus \lambda>n-1.$ Then 
    \begin{itemize}
        \item{Difficulty level 1:}  $\Pi_x(\lambda)$ is absolutely continuous for all $x\in \mathbb{R}^n\setminus \supp(\lambda).$
        \item{Difficulty level 2:}  $\Pi_x(\lambda)$ is in $L^2(S^{n-1})$ for all $x\in \mathbb{R}^n\setminus \supp(\lambda).$
        \item{Difficulty level 3:}  $\Pi_x(\lambda)$ is in $L^p(S^{n-1})$ for all $p>1$ and all $x\in \mathbb{R}^n\setminus \supp(\lambda).$
        \item{Difficulty level 4:}  $\Pi_x(\lambda)$ is absolutely continuous with a upper semi continuous density function for all $x\in \mathbb{R}^n\setminus \supp(\lambda).$
    \end{itemize}
\end{conj}
\begin{rem}
It might be the case that the difficulty level 1,2,3 statements hold with $\lambda$ being a general self-similar measure with the open set condition. However, this seems to be too strong to hold. We are not able to find counterexamples either. 
\end{rem}
\begin{rem}
If instead of "for all $x\in\mathbb{R}^n\setminus\supp(\lambda)$", we require only "most of $x\in\mathbb{R}^n\setminus\supp(\lambda)$" then the situation is much more clear. In fact, Theorems \ref{thm: Or1}, \ref{thm: Or2} already give such a result. Stronger results establishing the nullity of the dimension of the exceptional sets can be found in \cite{KO, ShmerkinSolomyak}.
\end{rem}
Currently, even the difficulty level 1 problem is open. In this paper, we partially solve the above conjecture up to the difficulty level 3. For clearness, we state our result for a particularly chosen $\lambda.$ In the next section, we will provide general results.

\begin{thmx}\label{thm: A}
    Let $\lambda=\lambda_{p,D}$ be the missing digits measure on $\mathbb{R}$ with base $p=10^{10000}$ and the digit set
    \[
    D=\{1,2,\dots,10^{8000}\}.
    \]
    Then for each $x\notin \supp(\lambda)\times\supp(\lambda),$ $\Pi_x(\lambda\times \lambda)$ is absolutely continuous with a density function in $L^2(S^1).$
\end{thmx}

Let $n\geq 2$ be an integer. Let $\lambda$ be a compactly supported Borel measure on $\mathbb{R}^n.$ Suppose that $x\in\mathbb{R}^n\setminus \supp(\lambda).$ We see that if $\Pi_x(\lambda)$ is absolutely continuous, then $\Pi_x(\supp(\lambda))\subset S^{n-1}$ has positive Lebesgue measure, or equivalently, positive $\mathcal{H}^{n-1}$ measure. If $\Pi_x(\lambda)$ is absolutely continuous with a upper semi continuous density function, then $\Pi_x(\supp(\lambda))$ has non-empty interior. For more applications, see Section \ref{sec: Integers}.

\subsection{Linear Projection} Let $n\geq 2$ be an integer. Let $\theta\in S^{n-1}$. Let $H_\theta$ be the ($n-1$ dimensional)  linear subspace normal to $\theta.$ For each $y\in\mathbb{R}^n,$ there is a point $h_y\in H_\theta$ so that $y\in h_y+\theta\mathbb{R}.$ This $h_y$ is uniquely determined with respect to $y.$ Define the function
\[
P_{\theta}:y\in \mathbb{R}^n\to h_y\in H_\theta.
\]
Thus $P_\theta$ is the linear projection that maps $\mathbb{R}^n$ onto $H_\theta.$ The fibres of this map $P_\theta$ are affine lines with direction $\theta.$ Notice that $P_\theta$ is the same as $P_{-\theta}.$ More generally, it is possible to define linear projections from $\mathbb{R}^n$ onto $k$ subspaces as long as $1\leq k\leq n-1.$ We only treat $k=n-1$ in this paper. In this setting, one has the general Marstrand projection theorem \cite[Section 3.5]{BP}.

\begin{thm}
Let $n\geq 2$ be an integer. Let $\lambda$ be a Borel probability measure on $\mathbb{R}^n.$ Suppose that
\[
I_s(\lambda)=\int\int \frac{d\lambda(x)d\lambda(y)}{|x-y|^s}<\infty
\]
for some $s>n-1.$ Then for Lebesgue almost all $\theta\in S^{n-1},$ $P_\theta(\lambda)$ is absolutely continuous with respect to the Lebesgue measure on $H_\theta$. In particular, $P_\theta(supp(\lambda))$ has positive Lebesgue measure for almost all $\theta\in S^{n-1}.$
\end{thm}

As for radial projection, here we also have the following conjecture.

\begin{conj}\label{conj: linear projection}
Let $n\geq 2.$ Let $\lambda$ be a missing digits measure (or a Cartesian product of missing digits measures).  Suppose that $\Haus \lambda>n-1,$ then $P_\theta(\lambda)$ is absolutely continuous with a continuous density function for Lebesgue almost all $\theta\in S^{n-1}$.
\end{conj}
\begin{rem}
Here, if the reader would like to have a more challenging problem to solve, $\lambda$ can be replaced with a general self-similar measure with the open set condition. 
\end{rem}
Our method treating radial projections provides us with the following by-product which partially solves Conjecture \ref{conj: linear projection}.

\begin{thmx}\label{thm: B}
    Let $\lambda_1$ be the missing digits measure on $\mathbb{R}$ with base $p=10^{10000}$ and the digit set
    \[
    D=\{1,2,\dots,10^{5005}\}.
    \] Let $\lambda_2$ be the missing digits measure on $\mathbb{R}$ with base $p=11^{10000}$ and the digit set
    \[
    D=\{1,2,\dots,11^{5005}\}.
    \]
    Then for almost all $\theta\in S^1,$ $P_\theta(\lambda_1\times\lambda_2),$ $P_\theta(\lambda_1\times\lambda_1)$ are continuous functions. 
\end{thmx}
This implies, in particular, that $\supp{(\lambda_1)}+T_x\supp(\lambda_2)$ contains non-trivial interior for Lebesgue almost all $x\in\mathbb{R}.$ Here, $T_x$ is the map $y\in\mathbb{R}\to xy\in\mathbb{R}.$ In this direction, we recall the following verison of Palis' conjecture.
\begin{conj}
Let $A,B\subset [0,1]$ be two missing digits sets. Suppose that
\[
\Haus A+\Haus B>1.
\]
Then 'generically', $A+B$ contains non-trivial interior.
\end{conj}
\begin{rem}
This is not a very honest (and not even precise) version of Palis' conjecture. In fact, to be honest, we should say that, 'generically', if $A+B$ has a positive Lebesgue measure then it also has a non-trivial interior. For Cantor sets defined via non-linear dynamics, (the honest version of) Palis' conjecture was proved by Moreira and Yoccoz in \cite{MY}. See \cite{PS}, \cite{T1}, \cite{T2} for more discussions on (the honest and precise version of) Palis' conjecture for (linear) self-similar sets.
\end{rem}

In different settings, the meaning of the word 'generically' can differ from each other. In our situation with missing digits sets, we require that $A+T_x(B)$ contains a non-trivial interior for Lebesgue almost all $x\in\mathbb{R}.$ In this way, we see that Theorem \ref{thm: B} provides us with a special answer. We also suspect that the following much stronger result holds. This will provide us with a very definitive notion of 'genericity'.
\begin{conj}
Let $A,B\subset [0,1]$ be two missing digits sets. Suppose that the base $p_A$ of $A$ and $p_B$ of $B$ are such that
\[
\frac{\log p_A}{\log p_B}\notin\mathbb{Q},
\]
and that
\[
\Haus A+\Haus B>1.
\]
Then $A+T_x(B)$ contains non-trivial interior for all $x\in\mathbb{R}\setminus\{0\}.$ If $\log p_A/\log p_B\in\mathbb{Q},$ then for each $x\notin\mathbb{Q},$
$
A+T_x(B)
$ contains non-trivial interior.
\end{conj}
Although we are nowhere near this conjecture, the following result gives us hope.
\begin{thm}[Furstenberg's conjecture \cite{Fu2}]
Let $A,B\subset [0,1]$ be two missing digits sets. Suppose that the base $p_A$ of $A$ and $p_B$ of $B$ are such that
\[
\frac{\log p_A}{\log p_B}\notin\mathbb{Q},
\]
and that
\[
\Haus A+\Haus B>1.
\]
Then $A+T_x(B)$ has full Hausdorff dimension for all $x\in\mathbb{R}\setminus\{0\}.$ If $\log p_A/\log p_B\in\mathbb{Q},$ then for each $x\notin\mathbb{Q},$
$
A+T_x(B)
$ has full Hausdorff dimension.
\end{thm}
For the case when $\log p_A/\log p_B\notin\mathbb{Q},$ the result was proved in \cite{HS12}. For the other case, the result was proved in \cite{H14} as well as in \cite{Sh}.

\section{General result}
The general result depend on the notion of $\dim_{l^1}.$ See Section \ref{sec: L1}. For now, we remark that for missing digits measures $\lambda$ with large bases and simple digits sets,  $\dim_{l^1} \lambda$ is almost equal to $\Haus \lambda.$

\begin{thm}\label{thm: n1}
    Let $n\geq 2$ be an integer. Let $\lambda$ be a Borel probability measure on $\mathbb{R}^n.$
    
    Suppose that
    \[
    \dim_{l^1}\lambda>n-1.
    \]
    For each $x\in\mathbb{R}^n\setminus \supp(\lambda),$ if $\Haus\Pi_x(\lambda)=n-1$ then $\Pi_x(\lambda)$ is absolutely continuous. 
    
    Suppose that for an integer $p>1,$
    \[
    \dim_{l^1}\lambda>n-p^{-1}.
    \]
    Then for each $x\in\mathbb{R}^n\setminus \supp(\lambda),$ $\Pi_x(\lambda)$ is absolutely continuous with a density function in $L^{p}(S^{n-1}).$ In this case, the condition that $\Haus \Pi_x(\lambda)=n-1$ is not required to draw the conclusion.
\end{thm}
\begin{rem}
In some cases, the condition that $\Haus\Pi_x(\lambda)=n-1$ is automatic. For example, let $\lambda$ be a missing digits measure on $\mathbb{R}^2$ with $\Haus \lambda>1.$ Then \cite[Theorem 6.2]{Sh} implies that for each line $l\subset\mathbb{R}^2$ with irrational slope, for each $\epsilon>0,$ $\lambda(l^\delta)\ll \delta^{1-\epsilon}.$ This implies that for each $x\notin\supp(\lambda),$ each irrational $\theta\in S^1$ ($\theta=(\theta_1,\theta_2)$ with $\theta_1\theta_2\neq 0, ,\theta_1/\theta_2\notin\mathbb{Q}$), $\Pi_x(\lambda)(B_\delta(\theta))\ll \delta^{1-\epsilon}.$ Since $\Haus \lambda>1,$ $\lambda$ does not give positive measures for lines. This implies that $\Pi_x(\lambda)$ does not give positive measure for points. This implies further that the set of rational points in $S^1$ has zero $\Pi_x(\lambda)$ measure. This implies that $\Haus \Pi_x(\lambda)\geq 1-\epsilon$. Since $\epsilon>0$ can be arbitrarily chosen we see that $\Haus \Pi_x(\lambda)=1.$

Similarly (\cite[Theorem 7.2]{Sh}), suppose that $\lambda$ is a Cartesian product of two missing digits measures on $\mathbb{R}$ with bases $b_1,b_2$ such that $\log b_1/\log b_2\notin\mathbb{Q}.$ If $\Haus \lambda>1,$ then for each line $l\subset\mathbb{R}^2$ with non-trivial slope, i.e. not being parallel with the coordinate axis, then for each $\epsilon>0,$ $\lambda(l^\delta)\ll \delta^{1-\epsilon}.$ From here we again have $\Haus \Pi_x(\lambda)=1$ for all $x\notin\supp(\lambda).$

It is very likely that the aforementioned results in \cite{Sh} can be generalised to $\mathbb{R}^n$ for $n\geq 3.$ If this would be the case, then the condition $\Haus \Pi_x(\lambda)=n-1$ would be automatic if $\lambda$ is a missing digits measure or a Cartesian product of missing digits measures in lower-dimensional Euclidean spaces.
\end{rem}
It is not easy to obtain $\dim_{l^1}\lambda$ for general measures $\lambda.$ However, for missing digits measures, it is possible to obtain useful estimates. See Theorem \ref{thm: l1 bound}. Theorem \ref{thm: A} follows from Theorems \ref{thm: l1 bound} and the above theorem.

On the other hand, if $\lambda$ is a missing digits measure on $\mathbb{R}^n$ with $\Haus \lambda<n-1.$ Then $\Pi_x(\lambda)$ cannot be absolutely continuous for each $x\notin\supp(\lambda).$ Thus Theorem \ref{thm: n1} is almost sharp. In general, we have $\dim_{l^1}\lambda\leq \Haus \lambda$. In some cases, they can be almost the same. See Theorem \ref{thm: l1 bound}.

\begin{thm}\label{thm: n2}
    Let $n\geq 2$ be an integer. Let $\lambda$ be a Borel probability measure on $\mathbb{R}^n.$ Suppose that
    \[
    \dim_{l^1}\lambda>n-1.
    \]
    Then for almost all $\theta\in S^{n-1},$ $P_\theta(\lambda)$ is absolutely continuous with a continuous density function.
\end{thm}
Theorem \ref{thm: B} follows as a corollary of this Theorem and Theorem \ref{thm: l1 bound}.

\section{Preliminaries}
For more background materials in fractal geometry, see \cite{Fa}, \cite{Ma1}, \cite{Ma2}.
\subsection{Hausdorff dimension/measure}
Let $n\geq 1$ be an integer. Let $F\subset\mathbb{R}^n$ be a Borel set. Let $g: [0,1)\to [0,\infty)$ be a continuous function such that $g(0)=0$. Then for all $\delta>0$ we define the  quantity
\[
\mathcal{H}^g_\delta(F)=\inf\left\{\sum_{i=1}^{\infty}g(\mathrm{diam} (U_i)): \bigcup_i U_i\supset F, \mathrm{diam}(U_i)<\delta\right\}.
\]
The $g$-Hausdorff measure of $F$ is
\[
\mathcal{H}^g(F)=\lim_{\delta\to 0} \mathcal{H}^g_{\delta}(F).
\]
When $g(x)=x^s$ then $\mathcal{H}^g=\mathcal{H}^s$ is the $s$-Hausdorff measure and Hausdorff dimension of $F$ is
\[
\Haus F=\inf\{s\geq 0:\mathcal{H}^s(F)=0\}=\sup\{s\geq 0: \mathcal{H}^s(F)=\infty          \}.
\]

Let $\mu$ be a Borel probability measure on $\mathbb{R}^n.$ We define the upper/lower Hausdorff dimension of $\mu$ as follows,
\[
\overline{\Haus} \mu=\inf\{\Haus A: \mu(A)=1\},
\underline{\Haus} \mu=\inf\{\Haus A: \mu(A)>0\}.
\]
It is known that
\[
\underline{\Haus}\mu=\mathrm{essinf}_{x\sim \mu} \liminf_{\delta\to 0}\log \mu(B_\delta(x))/\log \delta.
\]
In particular, if $\underline{\Haus}\mu=s>0$ then $\mu(A)=0$ for all $A$ with $\Haus A<s.$

In this paper, we only consider the lower Hausdorff dimension. For this reason, we will write
\[
\Haus \mu=\underline{\Haus} \mu.
\]
\subsection{Missing digits sets/measures}\label{sec: missing}
We introduce the notion of missing digits sets. Let $n\geq 1$ be an integer. Let $p\geq 3$ be an integer. Let $D\subset\{0,\dots,p-1\}^{n}.$ Consider the  set
\[
K_{p,D}=\cl\{x\in\mathbb{R}^n: [p\{p^kx\}]\in  D,k\geq 0\},
\] 
where $\{x\},[x]$ are the component wise fractional part, integer part respectively of $x\in\mathbb{R}^n.$
Let $p_1,\dots,p_{\#D}$ be a probability vector, i.e. they are non-negative and sum to one. We can then assign each element in $D$ a probability weight. To be specific, one can first introduce an ordering on $D$ and assign the probabilities accordingly.  We can now construct the random sum
\[
S=\sum_{i\geq 1} p^{-i} \mathbf{d}_i
\]
where $\mathbf{d}_i\in D,i\geq 1$ are randomly and independently chosen from the set $D$ with the assigned probabilities.

If $p_1=\dots=p_{\#D}=1/\#D,$ the distribution of $S$ is a Borel probability measure supported on $[0,1]^n.$ We call this measure to be $\lambda_{p,D}.$ It is a Borel probability measure supported on $K_{p,D}\cap [0,1]^n.$ Moreover, it is AD-regular with exponent $\Haus K_{p,D}.$ We also write
\[
\dim_{l^1}K_{p,D}=\dim_{l^1}\lambda_{p,D}.
\]
\subsection{AD-regularity}\label{sec: AD}
Let $n\geq 1$ be an integer. Let $\mu$ be a Borel measure. Let $s>0$ be a number. We say that $\mu$ is $s $-regular, or AD-regular with exponent $s$ if there is a constant $C>1$ such that for all $x\in \supp(\mu)$ and all small enough $r>0$
\[
C^{-1} r^s\leq \mu(B_r(x))\leq C r^s,
\] 
where $B_r(x)$ is the Euclidean ball of radius $r$ and centre $x$. For an AD-regular measure $\mu$, the exponent can be seen as
\[
s=\Haus \supp(\mu)=\Haus \mu.
\]
Missing digits measure $\lambda_{p,D}$ in $\mathbb{R}^n$ are AD-regular measures with exponent
\[
s=\Haus \lambda_{p,D}=\Haus K_{p,D}=\frac{\log \#D}{\log p^n}.
\]
\subsection{Radial projections of measures}\label{sec: rad}
Let $n\geq 2$ be an integer. Let $\lambda$ is a compactly supported Borel probability measure on $\mathbb{R}^n.$ For each $x\in \mathbb{R}^n,$ recall the radial projection function
\[
y\in\mathbb{R}^n\setminus\{x\}\to \Pi_x(y)=\frac{y-x}{|x-y|}\in S^{n-1}.
\]
Assume that $x$ is not in the support of $\lambda.$ We now define the measure $\Pi_x(\lambda).$ 

Let $h$ be a continuous and real valued function on $S^{n-1}.$ Let $Ph$ be the function on $\mathbb{R}^n\setminus \{x\}$ such that $Ph(y)=h(\theta)$ if and only if $\Pi_{x}(y)=\theta.$ Then $\Pi_{x}(\lambda)$ can be defined via the equality
\[
\int_{S^{n-1}} h(\theta)d\Pi_{x}(\lambda)(\theta)=\int Ph(y)d\lambda(y)
\]
for all continuous function $h$ on $S^{n-1}.$ The measure $\Pi_{x}(\lambda)$ is a compactly supported Borel probability measure on $S^{n-1}$.

Alternatively, it is possible to define $\Pi_x(\lambda)$ by differentiation. Let $\theta\in S^{n-1}.$ Let $\delta>0.$ Let $B_\delta(\theta)\subset S^{n-1}$ be the metric ball (Euclidean metric) of radius $\delta.$ Define the function
\[
\theta\in S^{n-1}\to r_\delta(\theta)=\frac{\lambda(\Pi_x^{-1}(B_\delta(\theta)))}{L(B_\delta(\theta))}\in [0,\infty),
\]
where $L$ is the Lebesgue probability measure on $S^{n-1}$. Then $r_\delta$ determines a compactly supported Borel probability measure on $S^{n-1}.$ It is possible to check that $\Pi_x(\lambda)$ is equal to the weak limit $\lim_{\delta\to 0} r_\delta$.  

\subsection{Fourier norm dimensions}\label{sec: L1}
Let $n\geq 1$ be an integer. Let $\lambda$ be a compactly supported Borel probability measure on $\mathbb{R}^n.$ Consider the Fourier transform
\[
\hat{\lambda}(\xi)=\int_{\mathbb{R}^n} e^{-2\pi i (x,\xi)}d\lambda(x),
\]
where $(.,.)$ is the standard Euclidean bilinear form. 
\begin{defn}
Let $p>0.$ We define
	\[
	\dim_{l^p}\lambda=\sup\left\{s>0: \sup_{\theta\in [0,1]^n}\sum_{|\xi|\leq R,\xi\in\mathbb{Z}^n} |\hat{\lambda}(\xi+\theta)|^p \ll R^{n-s}\right\}.
	\] 
\end{defn}
With the help of the Cauchy-Schwarz inequality, it is possible to show that
\[
\frac{\dim_{l^1}\lambda}{2}\leq \dim_{l^2}\lambda.
\]
Moreover, we have for each AD-regular measure $\lambda$
\[
\dim_{l^2}\lambda=\Haus \mu=\Haus \supp(\lambda).
\]
Furthermore, let $n\geq 1$ be an integer. Let $\lambda_1,\dots,\lambda_n$ be a Borel probability measure on $\mathbb{R}.$ The $n$-fold Cartesian product $\lambda'=\lambda_1\times\dots\times \lambda_n$ satisfies
\[
\dim_{l^1}\lambda'\geq \dim_{l^1}\lambda_1+\dots+\dim_{l^1}\lambda_n.
\]
We prove the following lemma.
\begin{lma}\label{lma: l1 integral}
Let $n\geq 1$ be an integer. Let $\lambda$ be a Borel probability measure on $\mathbb{R}^{n}$ with $\dim_{l^1}\lambda=s\in (0,n).$ Then for each $\sigma>n-s,$ we have
\[
\int_{\mathbb{R}^n} \frac{|\hat{\lambda}(\xi)|}{|\xi|^{\sigma}+1}d\xi<\infty.
\]
\end{lma}
\begin{proof}
From the definition of $\dim_{l^1}\lambda$ we see that for each $\epsilon>0,$
\[
\int_{2^{k-1}\leq |\xi|\leq 2^k}|\hat{\lambda}(\xi)|d\xi\ll 2^{k(n-s+\epsilon)}
\]
as $k\to\infty.$ Thus we see that
\[
\int_{2^{k-1}\leq |\xi|\leq 2^k} 2^{-k(n-s+2\epsilon)}|\hat{\lambda}(\xi)|d\xi\ll 2^{-k\epsilon}.
\]
This implies that
\[
\int_{\mathbb{R}^n} \frac{|\hat{\lambda}(\xi)|}{|\xi|^{n-s+2\epsilon}+1}d\xi<\infty.
\]
Since $\epsilon>0$ can be arbitrarily chosen, the result concludes.
\end{proof}
The next result was proved in \cite{ACVY} (for $n=1$) and \cite{Y22} (for $n\geq 1$). 
\begin{thm}\label{thm: l1 bound}
	Let $n\geq 1$ be an integer. The following results hold.
	\begin{itemize}
		\item{1:} Let $t\geq 1$ be an integer. We have
		\[
		\liminf_{p\to\infty,\#D\geq p^n-t} \dim_{l^1} \lambda_{p,D}=n.
		\]
		In particular, for each number $\epsilon>0,$ as long as $p$ is large enough, $\dim_{l^1}\lambda_{p,D}>n-\epsilon$ holds for each $D$ with $\#D=p^n-1.$
		\item{2:} For each integer $p\geq 4,$ let $D\subset\{0,\dots,p-1\}^n$ be a 'rectangle', i.e. a set of form $[a_1,b_1]\times [a_2,b_2]\dots [a_n,b_n]\cap \{0,\dots,p-1\}^n.$ Then we have \footnote{The base of $\log$ in this paper is $e$.}
		\begin{align*}\label{T}
		\dim_{l^1}\lambda_{p,D}\geq \Haus \lambda_{p,D}-\frac{n\log\log p^2}{\log p}.
		\end{align*}
	\end{itemize}
\end{thm}
This result only provides crude bounds of $\dim_{l^1}\lambda$ for missing digits measures $\lambda.$ See \cite{ACVY}, \cite{Y21}, \cite{Y22} for more detailed estimates and more applications.
\subsection{Asymptotic symbols}
We use the Vinogradov ($\ll,\gg,\asymp$) notations:

Let $f(\delta), g(\delta)$ be two real valued quantities depending on $\delta>0.$ Then
\begin{itemize}
    \item $f\ll g$ if $|f(\delta)|\leq C|g(\delta)|$ for a constant $C>0$ and all $\delta>0.$
    
    \item $f\asymp g$ if $f\ll g$ and $g\ll f$.
\end{itemize}
\section{Bump functions on rectangles}\label{sec: bump}
Let $\phi: \mathbb{R}\to[0,\infty)$ be a non-zero Schwartz function such that $\hat{\phi}$ is supported on $[-1,1]$ and $0\leq \hat{\phi}(\xi)\leq 1$ for $\xi\in [-1,1]$. Then there is a constant $c=c_\phi>0$ such that for all $x\in\mathbb{R},$
\[
0\leq \phi(x)\leq c.
\]
Moreover, we have $\phi(0)>0.$ Thus, there are numbers $c',c''>0$ such that whenever $x\in [-c',c']$,
\[
\phi(x)>c''.
\]
To find such a non-zero function $\phi,$ we can first choose a function $\phi'$ supported on $[-1/2,1/2]$ and $0\leq \phi'(\xi)\leq 1$ for $\xi\in [-1/2,1/2].$ We can also assume that $\phi'$ is an even function. Then we can choose $\phi=\hat{\phi'*\phi'}.$

Let $\delta>0.$ Let $\phi_\delta(x)=\delta^{-1}\phi(x/\delta).$ Then $\hat{\phi}_\delta$ is supported on $[-1/\delta,1/\delta]$ and $0\leq \hat{\phi}(\xi)\leq 1$ for all $\xi\in\mathbb{R}$. Of course, we also have
\[
0\leq \phi(x)\leq c/\delta
\]
for all $x\in\mathbb{R}$. Next, for $x\in [-c'\delta,c'\delta],$ we have
\[
\phi(x)\geq \delta^{-1}c''.
\]

Let $n>1$ be an integer. Let $r_1,\dots,r_n>0$ be positive numbers. Let $R_{r_1,\dots,r_n}$ be the rectangle
\[
R_{r_1,\dots,r_n}=[-r_1,r_1]\times\dots\times [-r_n,r_n].
\]
Consider the function $\phi_{r_1,\dots,r_n},$
\[
x=(x_1,\dots,x_n)\to \phi_{r_1,\dots,r_n}(x)=\prod_{i=1}^n \phi_{r_i}(x_i).
\]
Then $\hat{\phi}_{r_1,\dots,r_n}$ is supported and with norm at most one on the rectangle
\[
\tilde{R}_{r_1,\dots,r_n}=[-r^{-1}_1,r^{-1}_1]\times\dots\times [-r^{-1}_n,r^{-1}_n].
\]
Moreover, we have that for all $x\in\mathbb{R}^n,$
\[
0\leq \phi_{r_1,\dots,r_n}(x)\leq c^n r^{-1}_1r^{-1}_2\times r^{-1}_n
\]
and for all $x\in c'R_{r_1,\dots,r_n}= [-c'r_1,c'r_1]\times\dots\times [-c'r_n,c'r_n],$ we have
\[
\phi(x)\geq {c''}^{n}/(r_1\dots r_n).
\]

Let $R\subset\mathbb{R}^n$ be a rectangle with side lengths $2r_1,\dots,2r_n.$ Let $\tilde{R}$ be the rectangle centred at the origin with side lengths $2r^{-1}_1,\dots,2r^{-1}_n$ and for each $i\in\{1,\dots,n\},$ the side of $\tilde{R}$ with length $r^{-1}_i$ is parallel with the side of $R$ with length $r_i.$ It is possible to find a rotation $g\in\mathbb{O}(n)$ and a translation $t\in\mathbb{R}^n$ such that $g(R_{r_1,\dots,r_n})+t=R.$ In this case, the function $\phi_R(x)=\phi_{r_1,\dots,r_n}(gx+t)$ has the property that  that $\hat{\phi}_R$ is supported and with norm at most one on the rectangle $\tilde{R}.$ Moreover, we have that for all $x\in\mathbb{R},$
\[
0\leq \phi_R(x)\leq c^nr^{-1}_1\dots r^{-1}_n
\]
and for $x\in c'R,$ we have
\[
\phi_R(x)\geq {c''}^{n}/(r_1\dots r_n).
\]
We note that $\hat{\phi}_R$ may not be real valued. 

\section{Proof of Theorem \ref{thm: n1}: the $L^p, p\geq 2$ case}
We prove the second conclusion of Theorem \ref{thm: n1}. The proof of the first conclusion (case $p=1$) is of a different nature and will be provided in the next section. 

The proof in this section is divided into 4 steps. In the first two steps. We provide the proof for the case $n=p=2.$ In this way, it is easier to picture the underlying ideas. In Step 3, we will provide the proof for larger $n.$ Finally, in Step 4, we extend the proof to deal with the cases when $p\geq 2.$
\subsection*{Step 1}
Let $\lambda$ be an $AD$-regular probability measure with $\dim_{l^1}\lambda>3/2.$ Let $K$ be the support of $\lambda$. Consider the radial projection $\Pi_x(\lambda)$ of the measure $\lambda$ at $x\in\mathbb{R}^2.$ Without loss of generality, we can assume that $x=(0,0)$ and that $K$ is contained in the ball centred at $(0,0)$ with radius one. Moreover, $(0,0)$ is not in the support of $\lambda.$ We now write $\Pi(\lambda)$ instead of $\Pi_x(\lambda).$

Let $\delta>0.$ For each $\theta\in S^1,$ let $T^\delta_\theta$ be the rectangle centred at the origin of side lengths $2,2\delta$ whose long side is pointing towards the direction $\theta$, i.e. the long side has a direction vector parallel with $\theta.$ Consider the function $f_\delta$
\[
\theta\in S^1\to f_\delta(\theta)=\frac{\lambda(T^\delta_\theta)}{\delta}\in [0,\infty).
\]
Since $\lambda$ is $AD$-regular with dimension larger than one, we see that $\lambda$ does not give positive measures to lines. This implies that $f_\delta$ is continuous. Our goal is to show that
\begin{align}\label{eqn: toshow}
\int_{S^1} f^2_\delta(\theta)d\theta\ll 1.
\end{align}
This implies that there is a sequence $\delta_i\to 0$ and a function $f\in L^2(S^1)$ such that $f_{\delta_1}\to f$ weakly (in the separable Hilbert space $L^2(S^1)$). 

The function $f_\delta$ is not the same as $r_\delta$ defined in Section \ref{sec: rad}. Since $\lambda$ is compactly supported and $(0,0)$ is not in the support of $\lambda,$ we see that $\lambda(T^\delta_\theta)$ is almost the same as $\lambda(\Pi^{-1}(B_\delta(\theta)))$. More precisely, there is a constant $0<C<1$ depending on $\lambda$ such that for all $\theta\in S^1,$
\[
\Pi^{-1}(B_{C\delta}(\theta))\subset T^\delta_\theta,
\]
and on the other direction, we have (as the support of $\lambda$ omit a ball centered at $(0,0)$)
\[
\lambda(T^\delta_\theta)\leq \lambda(\Pi^{-1}(B_{C^{-1}\delta}(\theta))).
\]
This implies that
\[
f_\delta(\theta)\ll r_\delta(\theta)\ll f_\delta(\theta)
\]
for all $\theta$ and the implicit constants in $\ll$ depend only on $\lambda.$ If (\ref{eqn: toshow}) holds, then 
\[
\int_{S^1} r^2_\delta(\theta) d\theta\ll 1
\]
as well. As above, we see that the limit $\lim_{\delta}r_\delta$ exists both in distribution and $L^2.$ This implies that $\Pi(\lambda)$ is absolutely continuous with respect to the Lebesgue measure and has an $L^2$ density function.
\subsection*{Step 2}
Recall the bump function $\phi_{{c'}^{-1}T^\delta_\theta}.$ We see that
\[
\lambda(T^\delta_\theta)\ll \delta\int \phi_{{c'}^{-1}T^\delta_\theta}(x) d\lambda(x).
\]
This is because $\phi_{{c'}^{-1}T^\delta_\theta}$ takes value in $[0,(c/c')^2 /\delta]$ and for $x\in c' ({c'}^{-1}T^\delta_\theta)=T^\delta_\theta,$ we have
\[
\phi_{{c'}^{-1}T^\delta_\theta}(x)\geq {(c''/c')}^2 /\delta.
\]
Recall that $c,c',c''>0$ are constants depending on the choice of the bump function $\phi$ in Section \ref{sec: bump}. By using Parseval's theorem, we see that
\[
\int \phi_{{c'}^{-1}T^\delta_\theta}(x) d\lambda(x)=\int \hat{\phi}_{{c'}^{-1}T^\delta_\theta}(\xi)\hat{\lambda}(-\xi)d\xi.
\]
Notice that $\hat{\phi}_{{c'}^{-1}T^\delta_\theta}$ is supported and with norm at most one on the rectangle $\tilde{{c'}^{-1}T^\delta_\theta}.$ For convenience, we write $R^\delta_\theta=\tilde{{c'}^{-1}T^\delta_\theta}.$ The rectangle $R^\delta_\theta$ has side lengths $2c'$ and $2c'/\delta.$ The side with length $2c'$ is pointing at the direction $\theta.$ From here we see that
\[
\int \hat{\phi}_{{c'}^{-1}T^\delta_\theta}(\xi)\hat{\lambda}(-\xi)d\xi\leq \int_{R^\delta_\theta} |\hat{\lambda}(\xi)|d\xi.
\]
Then we see that
\[
f_\delta(\theta)\ll\int_{R^\delta_\theta} |\hat{\lambda}(\xi)|d\xi,
\]
where the implicit constant in $\ll$ depends on $\phi$ only. Thus we have
\[
\int_{S^1} f^2_\delta(\theta)d\theta\ll \int_{S^1}\int\int |\hat{\lambda}(\xi)||\hat{\lambda}(\xi')|d\xi d\xi' d\theta.
\]
We now consider the triple integral above. By using Fubini's theorem, we see that
\[
\int_{S^1}\int_{R^\delta_\theta}\int_{R^\delta_\theta} |\hat{\lambda}(\xi)||\hat{\lambda}(\xi')|d\xi d\xi' d\theta=\int_{R^\delta_\theta}\int_{R^\delta_\theta} h(\xi,\xi') |\hat{\lambda}(\xi)||\hat{\lambda}(\xi')|d\xi d\xi',
\]
where the function $h(\xi,\xi')$ is defined via
\[
h(\xi,\xi')=|\{\theta\in S^1: \xi,\xi'\in R^\delta_\theta\}|.
\]
Here $|A|$ for a measurable set $A\subset S^1$ is the Lebesgue measure of $A.$ It is possible to see that for each $\xi\in\mathbb{R}^2$
\[
|\{\theta\in S^1: \xi\in R^\delta_\theta\}|\ll \frac{1}{|\xi|},
\]
where the implicit constant in $\ll$ depends on $\phi.$ Thus we see that
\[
h(\xi,\xi')\ll \frac{1}{\max\{|\xi|,|\xi'|\}}\leq \frac{1}{|\xi|^{1/2}|\xi'|^{1/2}}.
\]
This implies that
\[
\int\int h(\xi,\xi') |\hat{\lambda}(\xi)||\hat{\lambda}(\xi')|d\xi d\xi'\ll \left(\int |\hat{\lambda}(\xi)||\xi|^{-1/2}d\xi\right)^2.
\]
Since $\dim_{l^1}\lambda>3/2,$ Lemma \ref{lma: l1 integral} implies that
\[
\int |\hat{\lambda}(\xi)||\xi|^{-1/2}d\xi<\infty.
\]
From here we see that
\[
\int_{S^1} f^2_\delta(\theta)d\theta\ll 1.
\]
This is what we wanted to show in (\ref{eqn: toshow}).
\subsection*{Step 3}
The general case follows with similar arguments although it may not be easy to picture the ideas. Now let $n\geq 3$ be an integer and $\lambda$ be a compactly supported Borel probability measure on the unit ball centred at the origin with $\dim_{l^1}\lambda>n-(1/2).$ For the radial projection $\Pi_x$, we assume that $x=(0,\dots,0).$ Moreover, we assume that $(0,\dots,0)$ is not in the support of $\lambda$. We again write $\Pi$ instead of $\Pi_x.$

Let $1/2>\delta>0$ and $\theta\in S^{n-1}.$ Now, let $T^\delta_\theta$ be the rectangle of side lengths $2,2\delta,\dots,2\delta$ centred at the origin with the side of length $2$ being parallel with $\theta.$ Consider the function
\[
\theta\in S^{n-1}\to f_\delta(\theta)=\frac{\lambda(T^\delta_\theta)}{\delta^{n-1}}\in [0,\infty).
\]
We want to show that
\[
\int_{S^{n-1}}f^2_\delta(\theta)d\theta\ll 1.
\]
This will conclude the result just as in the case when $n=2$.

Consider the bump function $\phi_{{c'}^{-1}T^\delta_\theta}.$ The Fourier transform of $\phi_{{c'}^{-1}T^\delta_\theta}$ is essentially supported on a rectangle $R^\delta_\theta$. The rectangle $R^\delta_\theta$ has side lengths $$2c',  2c'/\delta, \dots, 2c'/\delta$$ and it is centred at the origin. Moreover, the side with length $2c'$ is parallel with $\theta.$ Intuitively speaking, the rectangle $R^\delta_\theta$ is a 'thin plate' normal to $\theta.$ 

Next, with the same arguments as in the previous steps, it is possible to see that
\[
\int_{S^{n-1}}f^2_\delta(\theta)d\theta\ll \int_{R^\delta_\theta}\int_{R^\delta_\theta}h(\xi,\xi')|\hat{\lambda}(\xi)||\hat{\lambda}(\xi')|d\xi d\xi'
\]
where
\[
h(\xi,\xi')=|\{\theta\in S^{n-1}: \xi,\xi'\in R^\delta_\theta\}|.
\]
Similar to the case when $n=2,$ it is possible to check that
\[
|\{\theta\in S^{n-1}: \xi\in R^\delta_\theta\}|\ll\frac{1}{|\xi|}.
\]
Loosely speaking, this is because that 
\[
\{\theta\in S^{n-1}: \xi\in R^\delta_\theta\}\subset S^{n-1}
\]
is a stripe of width roughly $1/|\xi|.$ Thus we see that\footnote{We remark that the inequality here is rather crude. In fact, for ''generic' $\xi,\xi'$ we have
\[
h(\xi,\xi')\ll \frac{1}{|\xi||\xi'|}.
\]
Our upper bound corresponds to the worst-case scenario when the two stripes (one for $\xi$ and one for $\xi'$) have a large intersection, e.g. they can be the same stripe.
}
\[
h(\xi,\xi')\ll \frac{1}{\max\{|\xi|,|\xi'|\}}.
\]
This implies that
\[
\int\int h(\xi,\xi') |\hat{\lambda}(\xi)||\hat{\lambda}(\xi')|d\xi d\xi'\ll \left(\int |\hat{\lambda}(\xi)||\xi|^{-1/2}d\xi\right)^2.
\]
Since $\dim_{l^1}\lambda>n-(1/2),$ the result concludes as in the previous step.
\subsection*{Step 4} We now consider the case when $p>2.$ As before, the goal is to show that as $\delta\to 0,$
\[
\int_{S^{n-1}}f^p_{\delta}(\theta)d\theta\ll 1.
\]
The arguments in Step 3 (or Step 2) can be borrowed without problems. Notice that
\[
\int_{S^{n-1}}f^{p}_\delta(\theta)d\theta\ll \int_{R^\delta_\theta}\int_{R^\delta_\theta}\dots \int_{R^\delta_\theta} h(\xi_1,\xi_2,\dots,\xi_p)|\hat{\lambda}(\xi_1)||\hat{\lambda}(\xi_2)|\dots |\hat{\lambda}(\xi_p)|d\xi_1\dots d\xi_p,
\]
where
\[
h(\xi_1,\dots,\xi_p)=|\{\theta\in S^{n-1}: \xi_1,\dots,\xi_p\in R^{\delta}_\theta\}|\ll \frac{1}{\max\{|\xi_1|,\dots,|\xi_p|\}}.
\]
Since $\dim_{l^1}\lambda>n-p^{-1},$ we see that
\[
\int_{S^{n-1}}f^{p}_\delta(\theta)d\theta\ll \left(\int
\frac{1}{|\xi|^{1/p}}|\hat{\lambda}(\xi)|d\xi\right)^p\ll 1.
\]
This is what we wanted to show.

\section{Proof of Theorem \ref{thm: n1}: the $L^1$ case}\label{sec: planar}
We prove the first conclusion of Theorem \ref{thm: n1}. We no longer separate the proof of the special case $n=2$. However, with $n=2,$ the geometrical idea is more transparent.

Let $\lambda$ be a Borel probability measure on $\mathbb{R}^n$ with $s_1=\dim_{l^1}\lambda>n-1.$ We assume that $(0,\dots,0)\notin\supp(\lambda)$ and consider $\Pi(\lambda)=\Pi_x(\lambda)$ with $x=(0,\dots,0).$ Assume that $\Haus \Pi(\lambda)=n-1.$ The result for other points $x\in\mathbb{R}^n$ can be considered in a similar way.

Let $R>1.$ Consider the annulus 
\[
A_R=\{R\leq |\xi|\leq 2R\}.
\]
We see that for each $\epsilon>0,$ as $R\to\infty,$
\[
\int_{A_R} |\hat{\lambda}(\xi)|d\xi\ll R^{n-s_1+\epsilon}.
\]
For each $\theta\in S^{n-1},$ let $S_{1/R}(\theta)\subset S^{n-1}$ be the stripe of width $1/R$ 'normal' to $\theta.$ More precisely, it is the set
\[
S_{1/R}(\theta)=\{\theta'\in S^{n-1}: |(\theta,\theta')|\leq 1/R\}.
\]
Consider the set
\[
Q_{R}(\epsilon)=\left\{\theta\in S^{n-1}: \int_{\xi\in A_R, \xi/|\xi|\in S_{1/R}(\theta)}|\hat{\lambda}{(\xi)}|d\xi\geq R^{n-1-s_1+2\epsilon} \right\}.
\]

We claim that $Q_R(\epsilon)$ can be covered with $\ll R^{n-1-\epsilon}$ many balls of radius $4/R.$ To see this, cover $S^{n-1}$ with $\asymp R^{n-1}$ many balls of radius $2/R$ whose centres at least $1/R$-separated from each other. Furthermore, it is possible to arrange that the maximum multiplicity of those balls is less than $C_n>1,$ a constant which depends on $n$ only. Let $\theta_i,i\geq 1$ be a numeration of the centres of those balls. Consider the sum
\[
\sum_{i\geq 1} \int_{S_{1/R}(\theta_i)}|\hat{\lambda}(\xi)|d\xi=\int |\{i\geq 1: \xi\in S_{1/R}(\theta_i)\}||\hat{\lambda}(\xi)|d\xi.
\]
Observe that
\[
|\{i\geq 1: \xi\in S_{1/R}(\theta_i)\}|\ll R^{n-2}.
\]
To see this, we temporally assume (without loss of generality) that $\xi=(1,0,\dots,0).$ In this case, we need $\theta_i$ to have
\[
|(\xi,\theta_i)|\leq R^{-1}.
\]
Namely, we write this condition in coordinates ($\theta_i=(\theta_{i,1},\dots,\theta_{i,n})$) as
\[
|\theta_{i,1}|\leq R^{-1}
\]
and
\[
\sum_{j=1}^n \theta^2_{i,j}=1.
\]
We require that $\theta_i's$ are $1/R$-separated. From here, we can confirm the observation. Therefore we see that
\[
\sum_{i\geq 1} \int_{S_{1/R}(\theta_i)}|\hat{\lambda}(\xi)|d\xi\ll R^{n-2}R^{n-s_1+\epsilon}.
\]
Thus there are at most
\[
\ll R^{n-1-\epsilon}
\]
many $i$ with
\[
\int_{S_{1/R}(\theta_i)}|\hat{\lambda}(\xi)|d\xi\geq R^{n-1-s_1+2\epsilon}.
\]
This proves the claim.

Let $k\geq 1$ be an integer. From the claim, we see that $Q_{2^k}(\epsilon)$ can be covered with $\ll 2^{k(n-1-\epsilon)}$ many balls of radius $\asymp 1/2^{k}.$ Let $1>\rho_\epsilon>(n-1-\epsilon)/(n-1)$ be a positive number. We see that
\[
\sum_{k\geq 1} 2^{k(n-1-\epsilon)} (1/2^{k})^{(n-1)\rho_\epsilon}<\infty.
\]
From the convergence Hausdorff-Borel-Cantelli lemma (\cite{HS18}), we see that
\[
\limsup_{k\to\infty} Q_{2^k}(\epsilon)\subset S^{n-1}
\]
has zero $(n-1)\rho_\epsilon$-Hausdorff measure. This implies that
\[
\Haus \limsup_{k\to\infty} Q_{2^k}(\epsilon)\leq (n-1)\rho_\epsilon<n-1.
\]

Since $\Haus \Pi(\lambda)=n-1,$ we see that $\Pi(\lambda)(\limsup_{k\to\infty} Q_{2^k}(\epsilon))=0.$ In other words, for $\Pi(\lambda)$ almost all $\theta\in S^{n-1},$ there are at most finitely many $k\geq 1$ such that $\theta\in Q_{2^k}(\epsilon).$ Since $s_1>n-1,$ it is possible to choose $\epsilon>0$ such that $n-1-s_1+2\epsilon<0.$ Fix such a choice of $\epsilon$. Let $\theta\in S^{n-1}$ be such that $\theta\in Q_{2^k}(\epsilon)$ for at most finitely many $k.$ Let $P^{1/100}_\theta$ be the slab normal to $\theta,$ centred at the origin, with width $1/100,$
\[
P^{1/100}_\theta=\{x\in\mathbb{R}^n: |(\theta,x)|\leq 1/200\}.
\]
Then we see that
\begin{align}\label{eqn: L^1 Fourier}
\int_{P^{1/100}_\theta} |\hat{\lambda}(\xi)|d\xi=\int_{P^{1/100}_\theta\cap B_1(0)} |\hat{\lambda}(\xi)|d\xi+\sum_{k\geq 0}\int_{P^{1/100}_\theta\cap A_{2^k}} |\hat{\lambda}(\xi)|d\xi<\infty.
\end{align}
This is because for all large enough $k,$
\[
dir(P^{1/100}_\theta\cap A_{2^k})\subset S_{1/{2^k}}(\theta),
\]
where $dir(x)=x/|x|$ for $x\in\mathbb{R}^n\setminus\{0\}$ and $dir(A)=\{dir(x): x\in A\}\subset \mathbb{S}^{n-1}.$
Thus there is a constant $c'''>0$ (depending on the choice of $\phi$ in Section \ref{sec: bump}) such that
\[
\lambda(T^{c'''\delta}_\theta)/\delta^{n-1}\ll \int_{P^{1/100}_\theta}|\hat{\lambda}(\xi)|d\xi<\infty. 
\]
In other words (rescale $\delta$ if necessary), we have
\[
\limsup_{\delta\to 0}\lambda(T^{\delta}_\theta)/\delta^{n-1}<\infty.
\]
Recall the function $f_\delta(\theta)=\lambda(T^\delta_\theta)/(\delta^{n-1}).$ Consider the limit $\lim_{\delta\to 0}f_\delta$ as a measure. This limit measure is equivalent to the measure $\Pi(\lambda).$ We have seen that for $\Pi(\lambda)$ almost all $\theta\in S^{n-1},$ 
\[
\liminf_{\delta\to 0}f_\delta(\theta)\leq \limsup_{\delta\to 0} f_\delta(\theta)<\infty.
\]
Thus the limit measure $\lim_{\delta\to 0}f_\delta$ and $\Pi(\lambda)$ are absolutely continuous with respect to the Lebesgue measure on $S^{n-1}.$ See \cite[Theorem 2.11]{Ma2}. 

\section{Proof of Theorem \ref{thm: n2}}
Let $\theta\in S^{n-1}.$ By performing a rotation, let us assume that $\theta=(0,\dots,0,1).$ We consider the projection $P_{\theta}(\lambda).$ For doing this, observe that for each $\xi\in\mathbb{R}^{n-1},$
\[
\hat{P_\theta(\lambda)}(\xi)=\int_{\mathbb{R}^{n-1}} e^{-2 \pi i (\xi,x)}dP_\theta(\lambda)(x)=\int_{\mathbb{R}^{n}} e^{-2 \pi i (\xi',x')}d\lambda(x')=\hat{\lambda}(\xi'),
\]
where $\xi'$ is the unique element in $\{\xi\}\times\{0\}\subset\mathbb{R}^n.$ Let $\xi'_t$ be the unique element in $\{\xi\}\times\{t\}$ for $t\in\mathbb{R}.$ Consider the function
\[
J_{\xi}:t\to \hat{\lambda}(\xi'_t).
\]
Then we see that (write $x'=(x'_1,\dots,x'_{n})$)
\[
J'_\xi(t)=\int_{\mathbb{R}^{n}} e^{-2\pi i (\xi'_t,x')} (-2\pi i x'_{n})d\lambda(x').
\]
The equality holds because $\lambda$ is compactly supported. Thus we see that
\[
|J'_\xi(t)|\leq 2\pi \int_{\mathbb{R}^{n}} |x'_n| d\lambda(x')<\infty.
\]
Notice that the RHS above is independent with respect to $\xi\in\mathbb{R}^{n-1}.$ This implies that ($d_{n-1}$ is the Lebesgue measure on $\mathbb{R}^{n-1}\times \{0\}$, $d_n$ is the Lebesgue measure on $\mathbb{R}^n$)
\[
\int_{\mathbb{R}^{n-1}\times\{0\}}|\hat{\lambda}(\xi')|d_{n-1} (\xi')\ll \int_{\mathbb{R}^{n-1}\times [-1/100,1/100]} |\hat{\lambda}(\xi')|d_n(\xi').
\]
Thus as long as 
\[
\int_{\mathbb{R}^{n-1}\times [-1/100,1/100]} |\hat{\lambda}(\xi')|d_n(\xi')<\infty,
\]
we have, for $\theta=(0,\dots,1),$
\[
\int_{\mathbb{R}^{n-1}}|\hat{P_\theta(\lambda)}(\xi)|d\xi<\infty.
\]
It follows that $P_\theta(\lambda)$ has a absolutely integrable Fourier transform. This implies that the measure $P_\theta(\lambda)$ has a continuous density function. From (\ref{eqn: L^1 Fourier}), we see that there is a set $E\subset S^{n-1}$ with Hausdorff dimension smaller than $n-1$ such that as long as $\theta\notin E,$ after performing a suitable rotation so that $\theta$ becomes $(0,\dots,1),$ the rotated measure (still written as $\lambda$) satisfies,
\[
\int_{\mathbb{R}^{n-1}\times [-1/100,1/100]} |\hat{\lambda}(\xi')|d_n(\xi')<\infty.
\]
From here we conclude that $P_\theta(\lambda)$ has a continuous density function for all such $\theta\notin E.$ Since $E$ does not have full Hausdorff dimension in $S^{n-1},$ it also has zero Lebesgue measure. From here, the proof finishes. 
\section{Integers with restricted digits: Graham's problem}\label{sec: Integers}

The difficulty level 4 part of Conjecture \ref{conj} has a consequence in number theory. Let $k\geq 1$ be an integer. Let $2<b_1<b_2<\dots<b_k$ be integers. Let $D_1\subset\{0,\dots,b_1\},\dots,D_k\subset\{0,\dots,b_k\}$ be digit sets. Let $N^{D_1,\dots,D_k}_{b_1,\dots,b_k}$ be the set of positive integers whose base $b_1$ expansion contains digits only in $D_1,$ \dots, base $b_k$ expansion contains digits only in $D_k.$

We have the following result. See \cite{Y20}.
\begin{thm}
Assume the difficulty level 4 part of Conjecture  \ref{conj}. Suppose that 
\[
1,\frac{\log b_1}{\log b_2},\dots,\frac{\log b_1}{\log b_k}
\]
are $\mathbb{Q}$-linearly independent. If
\[
\sum_{i=1}^{k} \frac{\log |D_i|}{\log b_i}>k-1,
\]
then $N^{D_1,\dots,D_k}_{b_1,\dots,b_k}$ is infinite.
\end{thm}
The above theorem would answer a question of Graham which asks whether or not $N^{\{0,1\},\{0,1,2\},\{0,1,2,3\}}_{3,5,7}$ is infinite. In addition to the challenging difficulty level 4 problem, the $\mathbb{Q}$-linear independence of 
\[
1,\frac{\log b_1}{\log b_2},\dots,\frac{\log b_1}{\log b_k}
\]
is also not easy to be checked if $k\geq 3$. Those are the two heavy boulders in front of Graham's problem.\footnote{Of course, there might be a completely different approach to Graham's question that could bypass the two boulders.}

From Theorem \ref{thm: l1 bound} and Theorem \ref{thm: n1}, it is possible to draw some weaker conclusions. Let $t_1,\dots,t_k$ be numbers in $(0,1].$ Let $N^{D_1,\dots,D_k}_{b_1,\dots,b_k}(t_1,\dots,t_k)$ be the set of positive integers such that the base $b_1$ expansion of $[t_1n]$ contains digits only in $D_1,$ \dots, the base $b_k$ expansion of $[t_k n]$ contains digits only in $D_k.$

\begin{thm}
Let $4<b_1<b_2$ be integers with $\log b_1/\log b_2\notin\mathbb{Q}.$ Let $D_1, D_2$ be consecutive digit sets. If
\[
\sum_{i=1}^{2} \left(\frac{\log |D_i|}{\log b_i}-\frac{\log(\log b_i^2)}{\log b_i}\right)>1,
\]
then $N^{D_1,D_2}_{b_1,b_2}(t_1,t_2)$ is infinite for Lebesgue almost all $(t_1,t_2)\in (0,1]^2$.
\end{thm}
It is not difficult to find explicit examples. For instance, let $b_1=10^{10000}$ and $b_2=11^{10000}.$ Consider the digit sets
\[
D_1=\{1,\dots,10^{5005}\}, D_2=\{1,\dots,11^{5005}\}.
\]
Without the above theorem, it is not clear whether or not there exist $t,t'\in (0,1]$ such that
\[
N^{D_1,D_2}_{b_1,b_2}(t,t')
\]
is infinite. Of course, Conjecture \ref{conj} implies that the above set is infinite for all $t,t'\in (0,1]$ including the most interesting case $t=t'=1.$

\section{Acknowledgement}
HY was financially supported by the University of Cambridge and the Corpus Christi College, Cambridge. HY has received funding from the European Research Council (ERC) under the European Union's Horizon 2020 research and innovation programme (grant agreement No. 803711). HY has received funding from the Leverhulme Trust (ECF-2023-186).

\subsection*{Rights}

For the purpose of open access, the authors have applied a Creative Commons Attribution (CC-BY) licence to any Author Accepted Manuscript version arising from this submission.

\bibliographystyle{amsplain}

\end{document}